\documentclass[a4paper,10pt]{article}
\usepackage{amsfonts,graphics,amsmath,amsthm,amsfonts,amscd,amssymb}
\usepackage{amsmath,latexsym,multicol,mathrsfs,stmaryrd,color}
\usepackage{epsfig,url,pxfonts,accents}
\usepackage{flafter,multirow}
\usepackage{galois}
\usepackage{fancyhdr}
\usepackage{hyperref}
\usepackage{textcomp}

\hypersetup{colorlinks=true, linkcolor=black}

\newcommand{\ovline}[1]{%
  \vbox{%
    \hrule height 0.6pt
    \kern0.22ex
    \hbox{%
      \kern-0.05em
      \ifmmode#1\else\ensuremath{#1}\fi
      \kern0em
    }
  }
}

\makeatletter

\def\jobis#1{FF\fi
  \def\predicate{#1}%
  \edef\predicate{\expandafter\strip@prefix\meaning\predicate}%
  \edef\job{\jobname}%
  \ifx\job\predicate
}

\makeatother

\if\jobis{proposal}%
\else
\fi

\usepackage[all]{xy}

\DeclareMathOperator{\Supp}{Supp}

\DeclareMathOperator{\Hom}{Hom}
\DeclareMathOperator{\pic}{Pic}

 \numberwithin{equation}{subsection}
 \numberwithin{footnote}{subsection}

 \newtheorem{lem}[subsection]{Lemma}
 \newtheorem{prop}[subsection]{Proposition}
 \newtheorem{thm}[subsection]{Theorem}
 \newtheorem*{thmA}{Theorem A}
 \newtheorem*{thmK}{Theorem K}

{
    \newtheoremstyle{upright}%
        {8pt plus2pt minus4pt}%
        {8pt plus2pt minus4pt}%
        {\upshape}%
        {}%
        {\bfseries\scshape}%
        {}%
        {1em}%
        {}%
\theoremstyle{upright}

 \newtheorem{rem}[subsection]{Remark}

 \newtheorem{exam}[subsection]{Example}

}

\newcommand{\x}{\mathscr}
\newcommand{\f}{\mathfrak}
\newcommand{\cl}{\mathcal}
\newcommand{\C}{\mathbb C}

\newcommand{\oo}{\mathscr O}

\newcommand{\PP}{\mathbf P}

\newcommand{\Q}{\mathbb Q}
\newcommand{\R}{\mathbb R}

\newcommand{\Z}{\mathbb Z}
\newcommand{\spec}{\mathrm{Spec}\hspace{0.5mm}}

\newcommand{\lie}{\mathrm{Lie}}
\newcommand{\aut}{\mathrm{Aut}}

\newcommand{\supp}{\mathrm{\Supp}}

\newcommand{\dual}{\check{~}}
\newcommand{\rddown}[1]{\left\lfloor{#1}\right\rfloor} 
\newcommand{\gm}{\mathbb{G}_m}
\newcommand{\ga}{\mathbb{G}_a}
\newcommand{\cp}{\!\comp\!}

\begin{document}
\title{Homogeneous fibrations on log Calabi-Yau varieties}
\author{Jinsong Xu}
\date{}
\maketitle

\begin{abstract}
We prove a structure theorem for the Albanese maps of varieties with $\Q$-linearly trivial log canonical divisors. Our start point is the action of a nonlinear algebraic group on a projective variety.
\end{abstract}

\section{Introduction}
A log Calabi-Yau variety is a normal projective variety $X$ with an effective $\Q$-divisor $\Delta$ such that $K_X + \Delta$ is $\Q$-Cartier and numerically equivalent to zero. Usually, we assume the pair $(X, \Delta)$ has certain mild singularities. In his work toward Iitaka's conjecture on sub-additivity of the Kodaira dimensions, Kawamata proved the following structure theorem for varieties with numerically trivial canonical divisor:

\begin{thmK}\cite[Theorem 8.2]{ka}
Let $X$ be a projective variety over $\C$ with only canonical singularities, and assume the canonical divisor $K_X$ is numerically trivial. Then

(1) $K_X \sim_\Q 0$, i.e., $mK_X$ is a Cartier divisor and linear equivalent to zero for some integer $m >0$.

(2) The Albanese morphism $f \colon X \to A$ is surjective and has connected fibres.

(3) There is an \'{e}tale cover $\pi \colon B \to A$ such that $B \times_A X \cong B \times F$, where $F$ is a fibre of $f \colon X \to A$.
\end{thmK}

Later, this fundamental result was generalized to log Calabi-Yau varieties with Klt (Kawamata log terminal) singularities by Ambro:

\begin{thmA}\cite[Theorem 4.8]{a2}
Let $(X, \Delta)$ be a projective Klt pair over $\C$, and assume the log canonical divisor $K_X + \Delta$ is numerically trivial. Then

(1) $K_X + \Delta \sim_\Q 0$.

(2) The Albanese morphism $f \colon X \to A$ is surjective and has connected fibres.

(3) There is an \'{e}tale cover $\pi \colon B \to A$ such that $(X, \Delta) \times_A B \cong (F, \Delta_F) \times B$, where $F$ is a fibre of $f \colon X \to A$ and $\Delta_F$ is the restriction of $\Delta$ to $F$.
\end{thmA}

The idea of their proof is to first establish that the morphism $f \colon X \to A$ is ``isotrivial" so that the fibres vary in constant moduli, then they construct an isomorphism from an \'{e}tale neighborhood of the generic fibre to a product and extend this isomorphism over the whole. The techniques involved therein depend on some transcendental methods such as variation of Hodge structures.

In this paper, we give an alternative approach to Ambro's generalization (see Theorem \ref{result} below), which depends less on transcendental techniques. Indeed, apart from the standard characteristic zero tools (resolution of singularities, vanishing theorems, etc.), the only analytic input used in our argument is the Hodge duality for Klt spaces \cite{gkp}. Inspired by \cite{a1}, \cite{br1, br2} and \cite{ka}, we observed that on a log variety with torsion log canonical divisor, there is a duality between holomorphic one-forms and certain logarithmic vector fields. This enables us to make use of the actions of nonlinear algebraic groups to produce an isotrivial fibration. We also determine an explicit splitting \'etale cover $\pi \colon B \to A$. Our method allows us to see more details on the structure of Albanese morphism of irregular Calabi-Yau type varieties.
~\\

The property ``$K_X + \Delta \equiv 0$ implies $K_X + \Delta \sim_\Q 0$" has been established under more general assumption, see for example \cite{ckm}. We therefore start with $K_X + \Delta \sim_\Q 0$.

\begin{thm} \label{result}
Let $(X, \Delta)$ be a projective Klt pair over $\C$, and assume that $K_X + \Delta \sim_\Q 0$. Then

(1) The Albanese morphism $f \colon X \to A$ is a homogeneous fibration.

(2) There exists an abelian variety $B$, and an \'etale morphism $\pi \colon B \to A$ such that $B \times_A X \cong B \times F$, where $F$ is a fibre of $f$. Moreover, $X \cong B \times^G F$, where $G = \ker(\pi)$.
\end{thm}

We refer the reader to Section 3 and 4 for more precise statements and related definitions.

The organization of this paper is as follows: in Section 2, we discuss some preliminaries on reflexive sheaves and automorphism groups; in Section 3, we study algebraic group actions by abelian varieties, which serves as a key tool for our method; next, in Section 4, we first prove two key results on vector fields, then we prove the main result Theorem \ref{main}; finally, in Section 5, we illustrate some examples as applications and explanations.~\\

\textbf{Acknowledgements.} I would like to express my sincere thanks to Florin Ambro, Michel Brion, Yujiro Kawamata and Chenyang Xu for their invaluable suggestions and discussions. I am also grateful to Baohua Fu for warm encouragement and support. Finally, I would like to thank Osamu Fujino for useful discussion during my visit to Kyoto university in 2014.

This work is partially supported by NNSFC (No.11401358).

\section{Preliminaries}
\subsection{Notations}
We work over $k = \C$, the field of complex numbers.

A scheme always means a scheme over $k$.

A variety is an integral and separated scheme of finite type over $k$.

A log variety is a pair $(X, \Delta)$, where $X$ is a normal variety and $\Delta$ is an effective $\Q$-divisor such that $K_X + \Delta$ is $\Q$-Cartier.

For a Weil $\R$-divisor $D = \sum a_iD_i$, we set
$$
\lceil D \rceil := {\textstyle \sum} \lceil a_i \rceil D_i, \hspace{4mm} \lfloor D \rfloor := {\textstyle \sum} \lfloor a_i \rfloor D_i, \hspace{3mm} \mbox{ and } \hspace{1mm} \{D\} := D - \lfloor D \rfloor.
$$

We refer to \cite{km} for notions of Klt singularities et al.

\subsection{Reflexive sheaves}
Let $X$ be a normal variety, and let $\Delta \geq 0$ be a reduced Weil divisor. We consider $\Delta$ as a closed subscheme of $X$, endowed with reduced subscheme structure. Choose a big open set $j \colon U \hookrightarrow X$ such that $(U, \Delta|_U)$ is log smooth (i.e., $U$ is nonsingular and $\Delta|_U$ has simple normal crossing support). Define
$$
\Omega^{[p]}_X(\log\Delta) := j_*\Omega^p_U(\log(\Delta|_U)).
$$
This definition does not depend on the choice of the big open set $U$, and is called the \emph{sheaf of logarithmic $p$-forms} along $\Delta$. $\Omega^{[p]}_X(\log\Delta)$ is a reflexive sheaf containing $\Omega_{X}^{[p]} = j_*\Omega^p_U$. We set
$$
T_X(-\log\Delta) := \Omega^{[1]}_X(\log\Delta)\dual = \x{H}om(\Omega^{[1]}_X(\log\Delta), \oo_X).
$$
It is a reflexive subsheaf of the tangent sheaf $T_X := \x{H}om(\Omega^1_X, \oo_X)$. Sections of $T_X(-\log\Delta)$ are just $k$-derivations of the sheaf $\oo_X$ to itself that send the ideal $\cl{I}$ of $\Delta$ to $\cl{I}$.

\subsection{The automorphism group functor}
Let $X$ be a proper $k$-scheme, $i \colon \Delta \hookrightarrow X$ a closed subscheme. Let
$$
A_X \colon (\mathrm{Schemes}/k) \to \mathbf{Gr}
$$
be the automorphism functor for $X$ (cf. \cite[Section 3]{mo}). Given a $k$-scheme $S$, we define
$$
A_{X,\Delta}(S) = \{\sigma \in A_X(S)| \sigma \cp i_S \colon \Delta_S \to X_S \mbox{ factors through } \Delta_S \xrightarrow{i_S} X_S\},
$$
where $X_S, \Delta_S, i_S$ are base change of $X, \Delta, i$ respectively.

This defines a closed subfunctor $A_{X, \Delta}$ of $A_X$ from the category of $k$-schemes to the category of groups.

From \cite{mo} we know that $A_X$ is representable by a group scheme $\aut(X)$, locally of finite type over $k$. Standard argument shows that $A_{X, \Delta}$ is representable by a closed subgroup scheme $\aut(X, \Delta) \subset \aut(X)$. Closed points of $\aut(X, \Delta)$ correspond precisely to automorphisms $\sigma \colon X \to X$ that preserve the closed subscheme $i \colon \Delta \hookrightarrow X$.

Now let $X$ be a normal proper variety over $k$, and let $\Delta \geq 0$ be a reduced Weil \emph{divisor}. It is well known that $\lie(\aut(X))$, the tangent space of $\aut(X)$ at the identity morphism, can be identified canonically with the space of global vector fields $H^0(X, T_X)$. We wish to compute the tangent space of subgroup scheme $\aut(X, \Delta)$. Remember that we consider $\Delta$ as a \emph{reduced} closed subscheme of $X$.

\begin{lem} \label{lie}
The subspace
$$
H^0(X, T_X(-\log \Delta)) \subset H^0(X, T_X)
$$
is canonically isomorphic to $\lie(\aut(X, \Delta))$, the tangent space of $\aut(X, \Delta)$ at the identity morphism.
\end{lem}
\begin{proof}
An element in $\lie(\aut(X, \Delta)) = A_{X, \Delta}(\spec k[\varepsilon])$ is an automorphism of the $k[\varepsilon]$-algebra $\oo_X + \varepsilon\oo_X$ that maps $\cl{I} + \varepsilon\cl{I}$ to $\cl{I} + \varepsilon\cl{I}$, where $\cl{I}$ is the ideal of $\Delta$, and reduce to the identity morphism modulo $\varepsilon$. Therefore if $a \in \oo_X$, we have $\sigma(a) = a + \varepsilon\partial a$ for some $\partial a \in \oo_X$. It is easy to verify that $a \mapsto \partial a$ is a $k$-derivation of $\oo_X$ to itself and maps $\cl{I}$ to $\cl{I}$. Conversely, every $k$-derivation $\partial \in H^0(X, T_X(-\log \Delta))$ defines a $k[\varepsilon]$-automorphism $\sigma \colon \oo_X + \varepsilon \oo_X \to \oo_X + \varepsilon \oo_X$ by sending $a + \varepsilon b$ to $a + \varepsilon (\partial a + b)$. We verify readily that $\sigma \in \lie(X, \aut(X, \Delta))$.
\end{proof}

\section{Group actions by abelian varieties}
Following \cite{br2}, a morphism $f \colon X \to A$ from a projective variety $X$ to an abelian variety $A$ is called a \emph{homogeneous fibration} if $f_*\oo_X = \oo_A$ and $f$ is isomorphic as an $A$-scheme to its pull-back by any translation in $A$. We see immediately that the fibers of a homogeneous fibration are connected and isomorphic to each other.

The theorem below is of fundamental importance for algebraic group actions by abelian varieties. It is a generalization of the well-known Poincar\'e's complete reducibility theorem for abelian varieties. Similar results hold for actions of connected nonlinear algebraic groups (i.e., algebraic groups that are not linear), we refer to \cite{br1,br2,bsu} for further results.~\\

We write the group operation of an abelian variety \emph{multiplicatively} except in Example \ref{klt}.

\begin{thm} \label{rmb}
Let $B$ be an abelian variety, $\alpha \colon B \times X \to X$ be a faithful action on a normal projective variety $X$. Then

(1) There exists a homogeneous fibration $f \colon X \to A$, where $A$ is a quotient of $B$ by a finite subgroup, and $f$ is equivariant.

(2) Let $F$ be the fibre of $f$ over $1 \in A$. Then the following diagram is commutative
$$
\xymatrix{
B \times F \ar@/_/[ddr]_\alpha \ar@/^/[drr]^{p_1}
\ar@{.>}[dr]_\sigma \\
& B \times_A X \ar[d]^{\pi_B} \ar[r]
& B \ar[d]^\pi \\
& X \ar[r]^f & A }
$$
and the induced morphism $\sigma \colon B \times F \to B \times_AX$ is an isomorphism.

In particular, $\alpha \colon B \times F \to X$ is an \'{e}tale cover.

(3) Let $G = \ker(\pi)$, then $G$ acts on $F$ and $X \cong B \times^G F$.

\end{thm}

In statement (2), any other fibre will also work, we choose the one over the identity element $1 \in A$ for convenience. The symbol $B \times^G F$ means the quotient of $B \times F$ by the diagonal action of $G$:
$$
g(b, x) = (bg^{-1}, gx), \hspace{3mm} g \in G, ~(b, x) \in B \times F.
$$
We need a lemma

\begin{lem} \label{stab}
Keep the assumption in Theorem \ref{rmb}. For a general point $x \in X$, its stabilizer $B_x = \{g \in B| gx = x\} \subset B$ consists of only identity morphism.
\end{lem}
\begin{proof}
We know from \cite[Chapter 2, Proposition 2.2.1]{bsu} that $B_x \subset B$ is a finite group scheme. (In characteristic zero, it is a finite group.) Since $\#B_x$ varies upper semi-continuously in $X$, there exists a Zariski open subset $U \subset X$ on which $\#B_x$ is a constant integer, say $d$.

Claim: $d = 1$.

Let us assume $d > 1$. Note that $B_x \subset B[d]$ for all $x \in U$, where $B[d] \cong (\Z/d\Z)^{2 \dim B}$ is the $d$-torsion subgroup of $B$. Hence there are only finitely many distinct groups $B_x$ for $x \in U$, say $B_1, \cdots, B_r$. The corresponding sets are denoted by $U_1, \cdots, U_r$. As $U = \bigcup U_i$ is dense in $X$, we see that at least one of the $U_i$'s is dense in $X$, say, $U_1$. By assumption, $d > 1$, and $B$ acts faithfully on $X$, we can take $\mathrm{id}_X \neq g \in B_1$, then $gx = x$ for all $x \in U_1$, therefore $g = \mathrm{id}_X$. A contradiction.
\end{proof}

\noindent \emph{Proof of Theorem \ref{rmb}.}
Statement (1) is a reformulation of results of Rosenlicht, Nishi, Matsumura and Brion. Here we only sketch the idea of proof. Fix an ample line bundle $L$ on $X$. The group action $\alpha \colon B \times X \to X$ defines a line bundle $\x{L} := \alpha^*L \otimes p_2^*L^{-1}$ on $B \times X$, rigidified along $1 \times X$, hence induces a morphism from $X$ to the Picard scheme $g \colon X \to \pic(B)$. One can show that $g$ is equivariant, and maps $X$ \emph{surjectively} onto a connected component say, $A^\prime$ of $\pic(B)$, thanks to the ampleness of $L$. The desired homogeneous fibration is obtained by taking Stein factorization of $g \colon X \to A^\prime$. See \cite[Chapter 2, Theorem 2.2.2]{bsu} and \cite[Proposition 1.2]{br2} for more details.

(2) and (3): Since $f$ is $B$-equivariant, we have $f(gx) = \pi(g)f(x) = \pi(g)$ for $g \in B, x \in F$, this shows $f \alpha = \pi p_1$. Moreover, $\alpha$ is a surjective and finite morphism, which implies that $\sigma$ is surjective onto an irreducible component $Z \subset B \times_AX$. We proceed to show $Z = B \times_AX$, and $\sigma$ is an isomorphism.

Let $y \in X$ be a general point, and consider the equation
$$
bx = y, \hspace{2mm} b \in B, x \in F. \eqno{(1)}
$$
It has at least one solution $(b_0, x_0) \in B \times F$. For any $g \in G$, clearly we have $b_0g^{-1}gx_0 = y$. Conversely, if $bx = y$, then $b^{-1}b_0x_0 = x$, so that $g = b^{-1}b_0 \in G$, and $b = b_0g^{-1}$. Therefore $\{(b_0g^{-1}, gx_0)\}_{g\in G} \subset B \times F$ constitute \emph{all distinct solutions} of (1) by Lemma \ref{stab}. Their images in $B \times_A X$ lie in $Z$. This is true for a general $y \in X$, it follows that $Z = B \times_A X$ and $\deg(\alpha) = \#G = \deg(\pi_B)$. Therefore $\sigma$ is a proper birational and finite morphism between normal varieties, whence an isomorphism, and $X \cong B \times^G F$. \qed ~\\

The connected automorphism group of a nonsingular projective variety with nonnegative Kodaira dimension is an abelian variety. We will use the following extension of this simple fact:

\begin{lem} \label{abelian}
Let $(X, \Delta)$ be a proper log variety, and assume $K_X + \Delta$ is pseudo-effective. Let $G$ be a connected algebraic group, acting faithfully on $X$ such that $g(\supp(\Delta)) = \supp(\Delta)$ for every $g \in G$. Then we have

(1) $G$ is a semi-abelian variety if $(X, \Delta)$ has only log canonical singularities.

(2) $G$ is an abelian variety if $(X, \Delta)$ has only Klt singularities.
\end{lem}
\begin{proof}
The lemma is essentially proved in \cite[Theorem 7]{i} and \cite[Proposition 4.6]{a1}. We reproduce a proof here for the reader's convenience.

To show a connected algebraic group $G$ is a semi-abelian variety (resp. an abelian variety), it suffices to show that $G$ contains no one-dimensional linear group $\ga$ (resp. $\ga$ and $\gm$). In view of this, we may assume that $G$ is either $\ga$ or $\gm$. Note that $\ga$ (resp. $\gm$) can be naturally identified with an open subset of $\PP^1$.

Assume that $(X, \Delta)$ is log canonical. The closed subset $\mathrm{Sing}(X) \cup \supp(\Delta)$ is $G$-invariant. By Rosenlicht's cross section theorem \cite[Theorem 10]{r1}, there exists a $G$-invariant open subset $U \subset X \setminus (\mathrm{Sing}(X) \cup \supp(\Delta))$ and a $G$-equivariant isomorphism $U ~\tilde{\rightarrow}~ G \times V$, where $G$ acts on $G \times V$ on the first factor by translation. Let $V \subset Y$ be a completion of $V$ such that $Y$ is nonsingular and $Y \setminus V$ is a simple normal crossing divisor. Together with the natural inclusion $G \subset \PP^1$, we obtain a birational map $X \dashrightarrow \PP^1 \times Y$.

By Hironaka's resolution of singularities, there exists a proper nonsingular variety $W$ and a commutative diagram
$$
\xymatrix{
& W \ar[dr]^g\ar[dl]_f&\\
X \ar@{-->}[rr] & & \PP^1 \times Y
}
$$
such that $f$ is an isomorphism over $U$, $g$ is an isomorphism over $G \times V$. In particular, we have $g(f^{-1}(X \setminus U)) = \PP^1 \times Y \setminus G \times V$.

Write
$$
	K_W + \Delta_W^+ = f^*(K_X + \Delta) + \Delta_W^-,
$$
where $\Delta_W^+$ and $\Delta_W^-$ are effective $\Q$-divisors without common components. As $\Delta$ is effective, the divisor $\Delta_W^-$ is effective, so that $K_W + \Delta_W^+$ is pseudo-effective because by assumption $K_X + \Delta$ is pseudo-effective. Hence the divisor $K_{\PP^1 \times Y} + \Delta'$ is pseudo-effective as well, here $\Delta' := g_*(\Delta_W^+)$ is a boundary divisor supported on $\PP^1 \times Y \setminus G \times V$.

Let $F \simeq \PP^1$ be a general fibre of the projection $\PP^1 \times Y \to Y$. Then
$$
	K_{\PP^1} + \Delta_{\PP^1} = (K_{\PP^1 \times Y} + \Delta')|_{F}
$$
is pseudo-effective. In particular, $\deg \Delta_{\PP^1} \geq 2$. However, $\Delta_{\PP^1}$ is a boundary divisor supported on $\PP^1 \setminus G$. Thus $G$ cannot be the additive group $\ga$. Moreover, if $(X, \Delta)$ is Klt, then $G$ cannot be $\ga$ or the multiplicative group $\gm$. This completes the proof.
\end{proof}

\section{Proof of the main result}
On a nonsingular projective variety with trivial canonical bundle, the natural pairing between global holomorphic one-forms and vector fields is non-degenerate. We generalize this to log varieties with torsion log canonical divisor, by discovering an equality of dimensions of one-forms and vector fields that tangent to the boundary.

The following important result of Greb, Kebekus and Peternell \cite[Proposition 6.9]{gkp} will be used in proving Lemma \ref{am}. It is an extension of Hodge duality for nonsingular varieties to mildly singular spaces. Here we state it in a slightly different way.

\begin{prop} \label{hd}
	(Hodge duality for Klt spaces). Let $Z$ be a $d$-dimensional normal projective variety. Suppose that there exists an effective $\Q$-divisor $D$ on $Z$ such that $(Z, D)$ is Klt. Then for each integer $0 \leq p \leq d$, there is an anti-linear isomorphism $H^0(Z, \Omega_X^{[p]}) ~\tilde{\rightarrow}~ H^{p}(Z, \oo_X)$.
\end{prop}

\begin{lem} \label{am}
Let $Y$ be a nonsingular projective variety of dimension $d$, and let $T$ be a $\Q$-divisor such that $T \sim_\Q 0$ and $\{T\}$ has simple normal crossing support. Then we have a natural anti-$\C$-linear isomorphism of complex vector spaces
$$
H^0(Y, T_Y(-\log \lceil\{T\}\rceil)( K_Y + \lceil T \rceil)) \cong H^1(Y, \oo_Y( K_Y + \lceil T \rceil))\dual.
$$
\end{lem}
\begin{proof}
Let $N > 0$ be an integer such that $NT$ is integral and linearly equivalent to zero, so $NT = \mathrm{div}(\varphi)$ for some rational function $\varphi \in k(Y)$. Let $\pi \colon Z \to Y$ be the normalization of $Y$ in $k(Y)(\!\!\sqrt[N]{\varphi})$. Then $Z$ has quotient singularities \cite[Lemma 3.24]{ev}, hence it is Cohen-Macaulay and its dualizing sheaf is $\Omega_Z^{[d]}$.

By Proposition \ref{hd} and Serre duality \cite[Corollary 7.7]{h}, we have an anti-$\C$-linear isomorphism
$$
H^0(Z, \Omega_Z^{[d-1]}) \xrightarrow{\mbox{\tiny{HD}}} H^{d-1}(Z, \oo_Z) \xrightarrow{\mbox{\tiny{SD}}} H^1(Z, \Omega_Z^{[d]})\dual \eqno{(2)}
$$
and the composition map preserves the action of the Galois group of $\pi$. Moreover, push-forward of $\Omega^{[p]}_Z$ decomposes into the sum of eigensheaves \cite[Lemma 4.6]{a2}
$$
\pi_*\Omega^{[p]}_Z = \bigoplus_{i=0}^{N-1} \Omega_Y^p(\log \lceil \{iT\} \rceil)(\rddown{iT}), \hspace{3mm} 0 \leq p \leq d.
$$
Hence the isomorphism (2) induces an anti-$\C$-linear isomorphism from the eigen-subspace of $H^0(Y, \pi_*\Omega_Z^{[d-1]}) = H^0(Z, \Omega_Z^{[d-1]})$ associated to $\zeta$, to the eigen-subspace of $H^1(Y, \pi_*\Omega_Z^{[d-1]}) = H^1(Z, \Omega_Z^{[d]})$ associated to $\ovline\zeta = \zeta^{-1}$:
$$
H^0(Y, \Omega_Y^{d-1}(\log \lceil\{T\}\rceil)(\lfloor T \rfloor)) ~\tilde{\rightarrow}~ H^1(Y, \Omega_Y^d(\log \lceil\{T\}\rceil)(\lfloor T \rfloor))\dual. \eqno{(3)}
$$
Observe that the natural pairing of locally free sheaves
$$
	\Omega_Y^p(\log \lceil\{T\}\rceil) \otimes \Omega_Y^{d-p}(\log \lceil\{T\}\rceil) \to \Omega_Y^d(\log \lceil\{T\}\rceil) = \oo_Y(K_Y + \lceil\{T\}\rceil)
$$
is non-degenerate, from which we deduce that
$$
	\Omega_Y^{d-1}(\log \lceil\{T\}\rceil)(\lfloor T \rfloor) \cong T_Y(-\log \lceil\{T\}\rceil)(K_Y + \lceil \{T\} \rceil + \lfloor T \rfloor) = T_Y(-\log \lceil\{T\}\rceil)(K_Y + \lceil T \rceil)
$$
and
$$
	\Omega_Y^d(\log \lceil\{T\}\rceil)(\lfloor T \rfloor) = \oo_Y(K_Y + \lceil \{T\} \rceil + \lfloor T \rfloor) = \oo_Y(K_Y + \lceil T\rceil.
$$
The lemma then follows from (3).
\end{proof}

\begin{prop} \label{vector}
Let $(X, \Delta)$ be a projective log variety with only Klt singularities. Assume that $K_X + \Delta \sim_\Q 0$, then we have a natural anti-$\C$-linear isomorphism
$$
H^0(X, T_X(-\log\lceil\Delta\rceil)) \cong H^1(X, \oo_X)\dual.
$$
In particular, $h^0(X, T_X(-\log\lceil\Delta\rceil)) = h^1(X, \oo_X) = h^0(X, \Omega_X^{[1]})$.
\end{prop}
\begin{rem}
This result fails if $(X, \Delta)$ is not Klt. See Example \ref{keyexam}.
\end{rem}
\begin{proof}
Let $\mu \colon Y \to X$ be an equivariant resolution of $X$ with respect to $\mathrm{Sing}(X) \cup \supp (\Delta)$ (cf. \cite[Section 1.1]{a1}). Let $\mu^*(K_X + \Delta) = K_Y + D$. The $\Q$-divisor $T = - D - K_Y \sim_\Q 0$, and $\{T\}$ has simple normal crossing support. Lemma \ref{am} gives an anti-$\C$-linear isomorphism
$$
H^0(Y, T_Y(-\log \lceil\{ D \}\rceil)(\lceil -D \rceil)) \cong H^1(Y, \oo_Y(\lceil -D \rceil))\dual.
$$
By Kawamata-Viehweg vanishing, $\oo_Y(\lceil -D \rceil)$ is $\mu_*$-acyclic and therefore the right-hand side is isomorphic to $H^1(X, \oo_X)\dual$ since $(X, \Delta)$ is Klt and $\lceil -D \rceil \geq 0$ is exceptional over $X$. By \cite[Lemma 1.1]{a1}, the left-hand side equals $H^0(X, T_X(-\log \lceil \Delta \rceil))$, since $\mu$ is equivariant, and every vector field in $T_X(-\log \lceil \Delta \rceil)$ lifts to a vector field in $T_Y(-\log \lceil\{ D \}\rceil)$.
\end{proof}

Let $X$ be a normal projective variety, recall that its \emph{Albanese map} is by definition, the rational map $f = f_Y\mu^{-1} \colon X\dashrightarrow A$
$$
\xymatrix{
& Y \ar[dl]_\mu\ar[dr]^{f_Y} &\\
X \ar@{-->}[rr]^f& & A
}
$$
where $\mu \colon Y \to X$ is a resolution of singularities and $f_Y \colon Y \to A$ the Albanese morphism of $Y$. The rational map $f \colon X \dashrightarrow A$ does not depend on the choice of the resolution, and Kawamata showed in \cite{ka} that $f$ is everywhere defined if $X$ has rational singularities (e.g., $X$ has Klt or dlt singularities \cite[Chapter 5, Theorem 5.22]{km}). In that case we call it the \emph{Albanese morphism} of $X$. It is easy to see that $f \colon X \dashrightarrow A$ satisfies the universal mapping property into abelian varieties, i.e., any rational map from $X$ into an abelian variety factors uniquely through $f$.\\

Now we are in position to prove the main result.

\begin{thm}\label{main}
Let $(X, \Delta)$ be a projective log variety with only Klt singularities.  Assume the log canonical divisor $K_X + \Delta \sim_\Q 0$. Then

(1) The connected component of the group scheme $\aut(X, \lceil\Delta\rceil)$ containing the identity morphism, $B := \aut^0(X, \lceil\Delta\rceil)$, is an abelian variety of dimension $q = h^1(X, \oo_X)$.

(2) The Albanese morphism $f \colon X \to A$ is isomorphic to the homogeneous fibration induced by the action of $B$ on $X$, thus $A = B/G$ for some finite subgroup $G \subset B$.

(3) $B \times_A X \cong B \times F$, where $F$ is the fibre of $f$ over $1 \in A$. Moreover, $G$ acts on $F$ and $X \cong B \times^G F$.
\end{thm}
\begin{proof}
$B$ is a connected algebraic group, acts faithfully on $X$. Statement (1) follows from Lemma \ref{abelian}, Lemma \ref{lie} and Theorem \ref{vector}.

(2) Having in hand a faithful action on $X$ by the abelian variety $B$, Theorem \ref{rmb} produces a homogeneous fibration $g \colon X \to A^\prime$, where $A^\prime$ is an abelian variety of dimension $q$. This fibration factors through the Albanese morphism as
$$
\xymatrix{
& A \ar[dr]^\pi &\\
X \ar[ur]^f\ar[rr]^g & & A^\prime
}
$$
and we may assume $\pi$ preserve the identity elements, so that $\pi$ an isogeny of abelian varieties. Now note that $g$ is surjective and has connected fibres, it follows that $\pi$ is an isomorphism.

(3) This follows from Theorem \ref{rmb} (3).
\end{proof}

\section{Some examples}
Theorem \ref{main} enables us to classify certain log Calabi-Yau varieties with large irregularity. We illustrate the simplest case as an example.
\begin{exam} \label{klt}
 \emph{Classification of $d$-dimensional Klt log varieties $(X, \Delta)$ with $K_X + \Delta \sim_\Q 0$ and irregularity $q = h^1(X, \oo_X) = d - 1$}.

By Theorem \ref{main}, $X \cong B \times^G F$ is an \'etale quotient of $B \times F$, where $B$ is an abelian variety of dimension $q = d - 1$ and $F$ is a nonsingular curve. In particular, $X$ is nonsingular. The \emph{abelian} group $G \subset B$ acts faithfully on $F$ and thus can be considered as a subgroup of $\aut(F)$. We distinguish it into two cases.~\\

\noindent \emph{Case I.} $\Delta = 0$.

$F$ is an elliptic curve. Therefore the abelian group $G$ is a direct sum $H \oplus T$, where $H$ is a group of automorphisms of $F$ that preserves the group structure and $T$ is a group of translations in $F$. Note that by assumption the irregularity is $q = d - 1$, the group $H$ is nontrivial as long as $G$ is nontrivial. The classification of $H$ is well known (cf. \cite[Chapter VI]{b1}):
$$
\begin{array}{llll}
\hspace{6mm} F & \hspace{12mm} H & \mbox{Fixed points}\\
\mbox{arbitrary} & \hspace{6mm} \Z/2, \langle x \mapsto -x \rangle & F[2]\\
F_i = \C/(\Z \oplus \Z i) & \hspace{6mm} \Z/4, \langle x \mapsto ix \rangle, i = \sqrt{-1} & 0, \frac{1+i}{2}\\
F_\rho = \C/(\Z \oplus \Z \rho) & \hspace{6mm} \Z/3, \langle x \mapsto \rho x \rangle & 0, \frac{1- \rho}{3}, \frac{\rho - 1}{3}\\
& \hspace{6mm} \Z/6, \langle x \mapsto -\rho x \rangle, \rho^3 = 1, \rho \neq 1 & 0
\end{array}
$$
Moreover, as $T$ commutes with $H$, a translation in $T$ must preserve the fixed points of automorphisms in $H$. Up to isomorphism, the variety $X$ and the group $G$ belong to one of the following
$$
\begin{array}{llll}
& \hspace{3mm} X & \hspace{6mm} G = H \oplus T & \mbox{Action of } G \mbox{ on } F\\
\mbox{(I-1)} & B \times^G F & \hspace{2mm} \Z/2 & x \mapsto -x.\\
\mbox{(I-2)} & B \times^G F & \hspace{2mm} \Z/2 \oplus \Z/2 & x \mapsto -x, x \mapsto x + \varepsilon, \varepsilon \in F[2].\\
\mbox{(I-3)} & B \times^G F & \hspace{2mm} \Z/2 \oplus \Z/2 \oplus \Z/2 & x \mapsto -x, x \mapsto x + \varepsilon_1, x \mapsto x + \varepsilon_2,\\ &&&\varepsilon_1, \varepsilon_2 \in F[2], \varepsilon_1 \neq \varepsilon_2.\\
\mbox{(I-4)} & B \times^G F_i & \hspace{2mm} \Z/4 & x \mapsto ix.\\
\mbox{(I-5)} & B \times^G F_i & \hspace{2mm} \Z/4 \oplus \Z/2 & x \mapsto ix, x \mapsto x + (\frac{1+i}{2}).\\
\mbox{(I-6)} & B \times^G F_\rho & \hspace{2mm} \Z/3 & x \mapsto \rho x.\\
\mbox{(I-7)} & B \times^G F_\rho & \hspace{2mm} \Z/3 \oplus \Z/3 & x \mapsto \rho x, x \mapsto x + (\frac{1-\rho}{3}).\\
\mbox{(I-8)} & B \times^G F_\rho & \hspace{2mm} \Z/6 & x \mapsto -\rho x.
\end{array}
$$

When $d = 2$, the subcase (I-3) will not happen, and we recover the classification of bi-elliptic surfaces.~\\

\noindent \emph{Case II.} $\Delta \neq 0$.

Then $F = \PP^1$, and $G \subset B$ is a finite \emph{abelian} subgroup of $\mathrm{PGL}(2, \C)$. It is well known that up to conjugate, a finite subgroup $G \subset \mathrm{PGL}(2, \C)$ is isomorphic to one of the following groups
$$
\Z/r ~(r \geq 1), ~\f{D}_r ~(r \geq 2), ~\f{A}_4, ~\f{S}_4 ~\mbox{ and } ~\f{A}_5.
$$
We refer to \cite{b2} for a beautiful discussion of finite subgroups of $\mathrm{PGL}(2, K)$. In our case, $G$ is either $\Z/r$ or $\f{D}_2 = \Z/2 \oplus \Z/2$ since $G$ is abelian. In suitable affine coordinates, they can be explicitly represented as follows:
$$
\begin{array}{lll}
G = \Z/r, & \langle x \mapsto \zeta_rx \rangle, ~\zeta_r \mbox{ a primitive } r\mbox{-th root of } 1.\\
G = \Z/2 \oplus \Z/2, & \langle x \mapsto -x, x \mapsto x^{-1} \rangle.
\end{array}
$$

Moreover, we have $g^*\Delta = \Delta$ since $G$ is contained in the connected algebraic group $B$. Hence $g^*(\Delta|_F) = \Delta|_F$ for all $g \in G$, i.e., $\Delta|_F$ is a \emph{$G$-invariant divisor}. Here $``="$ means equality of Weil $\Q$-divisors. Set $\Gamma := \Delta|_F$. Up to isomorphisms, the variety $X$ is isomorphic to the quotient of $B \times (\PP^1, \Gamma) = (B \times \PP^1, B \times \Gamma)$ by $G$ listed below:
$$
\begin{array}{clll}
& X & G & \Gamma\\
\mbox{(II-1)} & B \times^G (\PP^1, \Gamma) & \Z/r & a_0\!\cdot\!0 + a_\infty\!\cdot\!\infty + \sum\limits_{t=1}^ka_t \!\cdot\! (\sum\limits_{s=1}^r \zeta_r^sx_t), \mbox{ where } 0 \leq a_0, a_\infty, a_t < 1 \\&&&\mbox{are rational numbers satisfying } a_0 + a_\infty + r\sum_t a_t = -2,\\&&& \mbox{and } \zeta_r^sx_t \in \PP^1\backslash\{0, \infty\} \mbox{ are distinct points}.\\
\mbox{(II-2)} & B \times^G (\PP^1, \Gamma) & \Z/2 \oplus \Z/2 & a_1\!\cdot\!(0 + \infty) + a_2\!\cdot\!(1 + (-1)) + a_3\!\cdot\!(i + (-i)) + \sum_tb_t\!\cdot\!(\sum_{g\in G}gx_t)\\&&&\mbox{where } 0 \leq a_1, a_2, a_3, b_t < 1 \mbox{ are rational numbers satisfying}\\&&& a_1 + a_2 + a_3 + 2\sum_tb_t = -1, \mbox{ and } gx_t \in \PP^1\backslash \{0, \infty, 1, -1, i, -i\}\\ &&&\mbox{are distinct pionts}.
\end{array}
$$
\end{exam}

In contrast to Example \ref{klt}, the case of log Calabi-Yau varieties with log canonical singularities are much more complicated. Moreover, Proposition \ref{vector} and Theorem \ref{main} generally fails for log canonical pairs.

\begin{exam} \label{keyexam}
Let $C$ be an elliptic curve, $\x{E} = \oo_C \oplus L$, where $L \in \mathrm{Pic}^0(C)$ is a line bundle that is of infinite order. Let $f \colon X = \PP(\x{E}) \to C$ be the ruled surface associated to $\x{E}$. $f$ has exactly two sections $D_0$ and $D_\infty$, corresponding to the projections $\x{E} \to L \to 0$ and $\x{E} \to \oo_C \to 0$ respectively.

The pair $(X, D_0 + D_\infty)$ is log canonical (or even plt), and satisfies
$$
D_0 \sim \oo_X(1), ~ D_\infty \sim \oo_X(1) - f^*L, ~ \mbox{ and } ~ K_X + D_0 + D_\infty \sim 0. \eqno{(4)}
$$
One verifies that

(1) $h^0(X, T_X(-\log(D_0 + D_\infty))) = 2 \neq h^1(X, \oo_X) = 1$, and

(2) There is no finite morphism from a proper curve $C^\prime \to C$ so that $C^\prime \times_C X$ splits into $C^\prime \times \PP^1$.

Though $X$ does not have splitting \'etale covers, the Albanese morphism $f \colon X \to C$ is \emph{induced by action of a non-linear algebraic groups}. Indeed, $G = \aut^0(X, \Delta)$ is a semi-abelian variety which is an extension of $C$ by $\gm$, corresponding to the extension class $L \in \Hom(\Z, \pic^0(C))$, and $G \subset X$ is an equivariant compactification with boundary divisor $D_0 + D_\infty$. Moreover, $X$ splits into a product after base change to the projection $G \to C$, and the Albanese morphism $f \colon X \to C$ is the homogeneous fibration induced by the action of $G$ on $X$.
\end{exam}

\begin{exam}
Here's a similar example, where the Albanese morphism is \emph{not induced by action of non-linear algebraic groups}. We consider again a ruled surface over an elliptic curve: let $C \subset \PP^2$ be a nonsingular cubic curve, $\x{E} = \oo_C \oplus \oo_C(-1)$, and $f \colon X = \PP(\x{E}) \to C$ the associated ruled surface, finally, $D_0$ (resp. $D_\infty$) the section corresponding to the surjections $\x{E} \to \oo_C(-1) \to 0$ (resp. $\x{E} \to \oo_C \to 0$).

We see again that $(X, D_0 + D_\infty)$ is log canonical and $K_X + D_0 + D_\infty \sim 0$. However, the Albansese morphism $f \colon X \to C$ is not induced by action of a non-linear algebraic group. Indeed, there is a birational contraction $\pi \colon X \to Y$, where $Y$ is a cone over $C$, that contracts $D_0$ to the vertex of $Y$. By Blanchard's lemma \cite[Proposition 4.2.1]{bsu}, any action of a connected algebraic group $G$ on $X$ descents to an action on $Y$, hence produces a homogeneous fibration from $Y$ to an abelian variety if $G$ is not linear. While the Picard number of $Y$ is $1$, there is no non-constant morphism into Abelian varieties.

\end{exam}

From the above two examples, we are interested in extending the our method to log Calabi-Yau varieties with log canonical singularities. In this situation, the expected group action occurring is a one by a semi-abelian variety, and the splitting \'etale cover will be replaced by a smooth cover, as we have seen in Example \ref{keyexam}.


\noindent Department of Mathematical Sciences,\\
\noindent Xi'an Jiaotong-Liverpool University,\\
\noindent No.111 Renai Road, Industrial Park, Suzhou, Jiangsu Province, China\\
\noindent e-mail: jinsong.xu@xjtlu.edu.cn
\end{document}